\documentclass[12pt]{amsart}
\usepackage{amssymb}
\usepackage{amsmath}
\usepackage{longtable}
\newcommand\g{{\mathfrak g}}
\newcommand\h{{\mathfrak h}}

\newcommand{\f}{\mathfrak{f}}

\renewcommand{\b}{\mathfrak{b}}
\newcommand\gl{\mathfrak{gl}}
\renewcommand{\a}{\mathfrak{a}}

\newcommand\m{\mathfrak m}

\newcommand\n{\mathfrak n}
\newcommand\q{\mathfrak q}
\newcommand\z{\mathfrak z}

\renewcommand{\t}{\mathfrak{t}}
\newcommand\codim{\operatorname{codim}}

\newcommand\im{\operatorname{im}}

\newcommand\Gr{\operatorname{Gr}}
\newcommand\Supp{\operatorname{Supp}}
\newcommand\K{{\mathbb{K}}}
\newcommand\X{\mathfrak X}
\newcommand\Q{\mathbb Q}

\renewcommand\P{\mathbb P}

\newcommand\D{\mathcal D}

\newcommand\Rad{\operatorname{R}}
\renewcommand\sl{\mathfrak{sl}}
\newcommand\so{\mathfrak{so}}
\renewcommand\sp{\mathfrak{sp}}

\newcommand\ord{\operatorname{ord}}
\newcommand\GL{\mathop{\rm GL}\nolimits}

\newcommand{\rank}{\mathop{\rm rk}\nolimits}

\newtheorem{Thm}{Theorem}
\newtheorem{Prop}{Proposition}[section]

\newtheorem{Lem}[Prop]{Lemma}
\theoremstyle{definition}

\newtheorem{defi}[Prop]{Definition}

\numberwithin{equation}{section}
\author{Ivan V. Losev}
\title{Demazure embeddings are smooth}
\thanks{Partially supported by A. Moebius foundation}
\begin{document}
\begin{abstract}
We prove the conjecture of M. Brion stating that the closure of the
orbit of a selfnormalizing spherical subalgebra in the corresponding
Grassmanian  is smooth.
\end{abstract}
\maketitle
\section{Introduction}
Throughout the paper the base field $\K$ is algebraically closed and
of characteristic zero. Let $G$ be a connected semisimple algebraic
group of adjoint type and $\g$ its Lie algebra.

A subalgebra $\h\subset \g$ is said to be {\it spherical} if there
is a Borel subalgebra $\b\subset\g$ such that $\b+\h=\g$. For
example, if $\h=\g^\sigma$ for an involutory automorphism $\sigma$
of $\g$ (in this case $\h$ is called symmetric), then $\h$ is
spherical, see \cite{Vust1}. For a symmetric  subalgebra
$\h\subset\g$ De Concini and Procesi proved in \cite{CP} that the
closure $\overline{G\h}$ of the orbit $G\h$ in the corresponding
Grassman variety is smooth. Earlier this fact in a few special cases
was proved by Demazure, \cite{Dem}. It was conjectured by Brion in
\cite{Brion} that the same is true for any spherical subalgebra $\h$
coinciding with its normalizer (note that a symmetric subalgebra
satisfies this condition). The variety $\overline{G\h}$ is called
the {\it Demazure embedding} of $G\h$.

Let us explain why the smoothness of Demazure embeddings is
important. There is a nice class of smooth projective $G$-varieties,
so called {\it wonderful varieties}, possessing many amazing
properties, see \cite{Timashev_rev}, Section 30, for a review. Knop
proved in \cite{Knop8} that a homogeneous space $G\h$ is embedded as
an open $G$-orbit into (a unique) wonderful variety provided $\h$ is
spherical and coincides with its normalizer. If $\overline{G\h}$ is
smooth, then it coincides with this wonderful variety. Brion's
results, \cite{Brion}, imply that the normalization of
$\overline{G\h}$ is wonderful.

The following theorem  is the main result of this paper.
\begin{Thm}\label{Thm_main}
Let $\h$ be a spherical subalgebra of $\g$ coinciding with its
normalizer. Then the Demazure embedding $\overline{G\h}$ is smooth.
\end{Thm}

 Let us note that this theorem was already proved
under certain restrictions on $\g$. In \cite{Luna_Dem} Luna gave the
proof  in the case when all simple ideals of $\g$ are $\sl$'s. In
\cite{BP} Luna's technique was extended to the case when any simple
ideal of $\g$ is isomorphic to $\sl_k$ or $\so_{2k}$. In this paper
we do not follow Luna's approach directly, however we use many
results
from \cite{Luna_Dem}.

Theorem \ref{Thm_main} is proved in Section 3. In Section 2 we recall some previosuly obtained
results related to wonderful varieties. 

{\bf Acknowledgements.} This paper was written during my visit to
Rutgers University, New Brunswick, in the beginning of 2007. I would
like to thank this institution and especially Professor F. Knop for
hospitality. I am also grateful to F. Knop for the formulation of
the problem and stimulating discussions.

\section{Preliminaries}\label{SECTION_prelim}
Below $G$ is a connected semisimple algebraic group of adjoint type,
$B$ its Borel subgroup, $T$ a maximal torus of $B$, $\Pi$ is the
system of simple roots of $G$ and $B^-$ is the Borel subgroup of $G$
containing $T$ and opposite to $B$. Let $\X(T)$ denote the character
lattice of $T$ (=the root lattice of $G$).

At first, let us recall the definition of a wonderful variety.
References are \cite{Luna4},\cite{Luna5},\cite{Luna_Dem},
\cite{Timashev_rev}, Section 30.

\begin{defi}\label{defi:1}
A $G$-variety $X$ is called  {\it wonderful}  if the following
conditions are satisfied:
\begin{enumerate}
\item $X$ is smooth and projective.
\item There is an open $G$-orbit $X^0\subset X$.
\item $X\setminus X^0$ is a divisor with normal crossings.
\item Let $D_1,\ldots,D_r$ be irreducible components of $X\setminus
X^0$. Then for any subset $I\subset \{1,\ldots,r\}$  the subvariety
$\bigcap_{i\in I}D_i\setminus\cup_{j\not\in I}D_j$ is a single
$G$-orbit.
\end{enumerate}
The number $r$ is called the rank of $X$ and is denoted by
$\rank_G(X)$.
\end{defi}

Note that $\bigcap_{i\in I}D_i$ is a wonderful $G$-variety of rank
$\#I$ for any $I\subset \{1,\ldots,r\}$.

It is known that a wonderful variety $X$ is spherical, that is $B$
has an open orbit on $X$.

Let us now establish some combinatorial invariants of a wonderful
$G$-variety $X$.

Note that $\cap_{i=1}^rD_r$ is a generalized flag variety. So there
is a unique $B^-$-stable point  $z\in \cap_{i=1}^r D_r$. The group
$T$ acts linearly on the normal space $T_zX/T_z(Gz)$. Note that
$T_zX/T_z(Gz)=\bigoplus_{i=1}^r T_zX/T_zD_i$. Let $\alpha_i$ denotes
the character of the action $T:T_zX/T_zD_i$. Let us fix an
$N_G(T)/T$-invariant scalar product on $\t^*(\Q)$. With respect to
this scalar product $\Psi_{G,X}:=\{\alpha_1,\ldots,\alpha_r\}$ is a
system of simple root of some root system in the linear span of
$\alpha_1,\ldots,\alpha_r$. The set $\Psi_{G,X}$ is called the {\it
system of spherical roots} of $X$. The definition of $\Psi_{G,X}$
given here agrees with that we use in \cite{unique}. Let $\X_{G,X}$
denote the sublattice in $\X(T)$ generated by $\Psi_{G,X}$. It is
known that $\X_{G,X}$ coincides with the set of all $\lambda\in
\X(T)$ such that there is a $B$-semiinvariant rational function
$f_\lambda$ on $X$ of weight $\lambda$ (determined uniquely up to
rescaling because $X$ is spherical). Choose a subset $\Psi_0\subset
\Psi_{G,X}$.  Put $\overline{X}^{\alpha_i}:=D_i,
\overline{X}^{\Psi_0}:=\cap_{ \alpha\in
\Psi_0}\overline{X}^{\alpha},
X^{\Psi_0}=\overline{X}^{\Psi_0}\setminus \cup_{\alpha\not\in
\Psi_0}\overline{X}^\alpha$.

Let $\D_{G,X}$ denote the set of all prime $B$-stable but not
$G$-stable divisors  on $X$
 (this definition differs slightly from that used in \cite{unique}). To each $D\in \D_{G,X}$ we
assign its stabilizer $G_D\subset G$, which is a parabolic subgroup
of $G$ containing $B$, and an element  $\varphi_D\in \X_{G,X}^*$
defined by $\langle\varphi_D,\lambda\rangle=\ord_D(f_\lambda)$. For
$\alpha\in \Pi$ by $P_\alpha$ we denote the minimal parabolic
subgroup of $G$ containing $B$ corresponding to the simple root
$\alpha$. Put $\D_{G,X}(\alpha)=\{D\in \D_{G,X}| P_\alpha\not\subset
G_D\}$.

The following propositions are due to Luna, see
\cite{Luna4},\cite{Luna5}.

\begin{Prop}\label{Prop:1}
For $\alpha\in \Pi(\g)$ exactly one of the following possibilities
takes place:
\begin{itemize}
\item[(à)] $\D_{G,X}(\alpha)=\varnothing$.
\item[(b)] $\alpha\in \Psi_{G,X}$. Here
$\D_{G,X}(\alpha)=\{D^+,D^-\}$ and
$\varphi_{D^+}+\varphi_{D^-}=\alpha^\vee|_{\a_{G,X}},
\langle\varphi_{D^\pm},\alpha\rangle=1$.
\item[(c)]  $2\alpha\in \Psi_{G,X}$. In this case $\D_{G,X}(\alpha)=\{D\}$ and $\varphi_{D}=\frac{1}{2}\alpha^\vee|_{\a_{G,X}}$.
\item[(d)] $\Q\alpha\cap \Psi_{G,X}=\varnothing, \D_{G,X}(\alpha)\neq\varnothing$. In this case
$\D_{G,X}(\alpha)=\{D\}$ and $\varphi_{D}=\alpha^\vee|_{\a_{G,X}}$.
\end{itemize}
\end{Prop}

We say that a root $\alpha\in \Pi$ is {\it of type} a) (or b),c),d))
if the corresponding possibility takes place for $\alpha$.

\begin{Prop}\label{Prop:3}
Let $\alpha,\beta\in \Pi(\g)$. If
$\D_{G,X}(\alpha)\cap\D_{G,X}(\beta)\neq\varnothing$, then exactly
one of the following possibilities takes place:
\begin{enumerate}
\item $\alpha,\beta$ are of type b) and
$\#\D_{G,X}(\alpha)\cap\D_{G,X}(\beta)=1$.
\item $\alpha,\beta$ are of type d),
$\langle\alpha^\vee,\beta\rangle=0,\alpha^\vee-\beta^\vee|_{\a_{G,X}}=0$,
and $\alpha+\beta=\gamma$ or $2\gamma$ for some
$\gamma\in\Psi_{G,X}$.
\end{enumerate}
Conversely, if $\alpha,\beta\in\Pi$ are such as in (2), then
$\D_{G,X}(\alpha)=\D_{G,X}(\beta)$.
\end{Prop}

\begin{Prop}\label{Prop:2}
Let $\alpha\in \Psi_{G,X},\beta\in \Pi\cap\Psi_{G,X}, D\in
\D_{G,X}(\beta)$. Then $\langle\varphi_D,\alpha\rangle\leqslant 1$
and the equality holds iff $\alpha\in \Pi, D\in \D_{G,X}(\alpha)$.
\end{Prop}
\begin{proof}
It follows from results of \cite{Luna4}, Subsection 3.5, (see also
\cite{Luna5}, Subsection 3.2) that in the proof one may replace $X$
with $\overline{X}^{\alpha,\beta}$. In this case everything follows
from the classification in \cite{Wasserman}.
\end{proof}

Now we are going to describe the localization procedure for
wonderful varieties.

Choose a subset $\Pi'\subset \Pi$. Let $M$ be the Levi subgroup of
$G$ corresponding to $\Pi'$. Put $G_{\Pi'}:=(M,M), Q^-=B^-M$. Then
there is a $G_{\Pi'}$-stable subvariety $X_{\Pi'}\subset
X^{\Rad(Q^-)}$ (where $\Rad(\cdot)$ denotes the radical) satisfying
the following conditions:
\begin{enumerate}
\item $z\in X_{\Pi'}$.
\item $X_{\Pi'}$ is a wonderful $G_{\Pi'}$-variety.
\item $\Psi_{G_{\Pi'},X_{\Pi'}}=\{\alpha\in \Psi_{G,X}| \Supp(\alpha)\subset
\Pi'\}$ (here and below  $\Supp(\alpha)$  stands for the set of all
$\beta\in\Pi$ such that the coefficient of $\beta$ in $\alpha$ is
nonzero).
\item For any $\alpha\in
\Pi'$ there is a bijection
$\iota:\D_{G_{\Pi'},X_{\Pi'}}(\alpha)\rightarrow\D_{G,X}(\alpha)$
such that $\varphi_D$ is the projection of $\varphi_{\iota(D)}$ to
$\X_{G_{\Pi'},X_{\Pi'}}$.
\item $GX_{\Pi'}=\overline{X}^{\Psi_{G,X}\setminus \Psi_{G_{\Pi'},X_{\Pi'}}}$.
\end{enumerate}
The $G_{\Pi'}$-variety $X_{\Pi'}$ is called the {\it localization}
of $X$ at $\Pi'$.

Proceed to the definition of Demazure morphisms.

Choose a point $x\in X^\varnothing$ and put $\h:=\g_x, d:=\dim\h$.
The Demazure morphism $\delta_X:X\rightarrow \Gr_d(\g)$ is defined
as follows: it maps $y\in X$ to the inefficiency kernel of the
representation of $\g_y$ in $T_xX/T_x(Gy)$, for example,
$\delta_X(x)=\h,\delta_X(z)$ is the intersection of the kernels of
all $\alpha\in\Psi_{G,X}$, where $\alpha\in \Psi_{G,X}$ is
considered as a character of a parabolic subalgebra $\g_z$. It is
known that the image of $\g_y$ in $\gl(T_xX/T_x(Gy))$ is a Cartan
subalgebra, so $\im\delta_X$ does lie in $\Gr_d(\g)$. Moreover,
$\im\delta_X=\overline{G\h}$. In \cite{Brion} Brion proved that
$\delta_X$ is the normalization morphism. So  $\overline{G\h}$ is
smooth iff $\delta_X$ is an isomorphism.  Further, Brion's result
implies the following statement.

\begin{Lem}\label{Lem:3.2}
Any element of $\im\delta_X$ is a spherical algebraic subalgebra of
$G$.
\end{Lem}

It is known that $\n_\g(\h)=\h$. Conversely, as Knop proved in
\cite{Knop8}, for any spherical subalgebra $\h$ coinciding with its
normalizer there is a wonderful variety $X$ such that
$X^\varnothing\cong G/N_G(\h)$. Note that $X^\varnothing=G/N_G(\h)$
has no nontrivial equivariant automorphisms. We say that a wonderful
variety $X$ is {\it rigid} if $X^\varnothing$ has no nontrivial
equivariant automorphisms. So theorem \ref{Thm_main} is equivalent
to the claim that $\delta_X$ is an isomorphism provided $X$ is
rigid.

The rigidity of $X$ can be expressed in terms of
$\Psi_{G,X},\D_{G,X}$. To state this result we need the following
definition.

\begin{defi}\label{defi:2}
An element $\alpha\in \Psi_{G,X}$ is said to  be  {\it
distinguished} if one of the following conditions holds:
\begin{enumerate}
\item $\alpha\in \Pi$ and  $\varphi_{D_1}=\varphi_{D_2}$ for different elements $D_1,D_2\in \D_{G,X}(\alpha)$.
\item There is a subset $\Sigma\subset \Pi$ of type
$B_k,k\geqslant 2,$ such that $\alpha=\alpha_1+\ldots+\alpha_k$ and
$\D_{G,X}(\alpha_i)=\varnothing$ for any $i>1$.
\item There is a subset $\Sigma\subset \Pi$ of type $G_2$
such that $\alpha=2\alpha_2+\alpha_1$.
\end{enumerate}
\end{defi}

Here $\alpha_i$ denote the simple roots of $B_k,G_2$ such that
$\alpha_k$ (for $B_k$) and $\alpha_2$ (for $G_2$) are  short.

The following proposition is a direct corollary of Theorem 2 from
\cite{unique}.

\begin{Prop}\label{Prop:5}
$X$ is rigid iff there are no distinguished elements in
$\Psi_{G,X}$.
\end{Prop}

 Now let
$\Pi'\subset \Pi$ and $\alpha\in \Psi_{G_{\Pi'},X_{\Pi'}}$. If
$\alpha$ is distinguished in $\Psi_{G,X}$, then it is distinguished
in $\Psi_{G_{\Pi'},X_{\Pi'}}$. The converse is not true: any
$\alpha\in \Pi\cap\Psi_{G,X}$ is distinguished in
$\Psi_{G_{\alpha},X_{\alpha}}$. However, if $\alpha$ is of type 2,3
in $\Psi_{G_{\Pi'},X_{\Pi'}}$, then it is of the same type in
$\Psi_{G,X}$.

\section{Proof of Theorem \ref{Thm_main}}\label{SECTION_reduction}
In this section $X$ is a rigid wonderful $G$-variety. Put
$\h=\delta_X(x)$ for some point $x\in X^\varnothing$. Let $\Pi^a$
denote the subset of $\Pi$ consisting of all roots of type a) for
$X$.

 The following two assertions were  proved in
\cite{Luna_Dem}, Section 3.

\begin{Lem}\label{Lem:3.1}
If  the restriction of $d_x\delta_X$ to the $T$-eigenspace
$(T_zX)_\gamma$ of weight $\gamma$ is injective for any $\gamma\in
\X(T)$, then $\delta_X$ is an isomorphism.
\end{Lem}

If $\gamma\not\in \Psi_{G,X}$, then $(T_zX)_\gamma\subset T_z(Gz)$.
Since $\delta_X:X\rightarrow \overline{G\h}$ is the normalization
morphism,  the restriction of $\delta_X$ to $Gz$ is an embedding.

\begin{Prop}\label{Prop:3.1}
Let $\alpha\in \Psi_{G,X}$ and $\Pi'$ be a subset of $\Pi$
containing $\Supp(\alpha)\cup\Pi^a$. The restriction of
$d_z\delta_X$ to $(T_zX)_\alpha$ is injective provided so is the
restriction of $d_z\delta_{X_{\Pi'}}$ to $(T_zX_{\Pi'})_\alpha$.
\end{Prop}

\begin{defi}\label{defi:3.2}
Let  $\alpha\in \Psi_{G,X}$. We say that $X$ is {\it critical for}
$\alpha$ if $\alpha$ is not distinguished in $\Psi_{G,X}$ but is
distinguished in $\Psi_{G_{\Pi'},X_{\Pi'}}$ for any $\Pi'\subsetneq
\Pi$ containing  $\Pi^a\cup\Supp(\alpha)$. 
\end{defi}

So we need  to prove the following claim:
\begin{itemize}
\item[(*)]  If $X$ is critical for $\alpha\in \Psi_{G,X}$,
then the restriction of $d_z\delta_X$ to $(T_zX)_\alpha$ is
injective.
\end{itemize}


\begin{Prop}
Let $X$ be critical for $\alpha\in \Psi_{G,X}$. (*) holds for $X$
provided $\alpha\not\in \Pi$.
\end{Prop}
\begin{proof}
It follows from \cite{Pezzini}, Theorem 3.4, that there is a simple
module $V$ and a $G$-equivariant morphism $\varphi:X\rightarrow
\P(V)$ such that the restriction of $\varphi$ to
$\overline{X}^{\alpha}$ is an embedding. So it remains to show that
there is a morphism $\psi: \delta_X(\overline{X}^\alpha)\rightarrow
\P(V)$ such that
$\psi\circ\delta_X|_{\overline{X}^\alpha}=\varphi|_{\overline{X}^\alpha}$.
Set $Y:=\delta_X(\overline{X}^\alpha),\h_1:=\delta_X(x)$ for some
$x\in X^{\alpha}, \h_0:=\delta_X(z)$.

Suppose, at first, that the character group of $G_x, x\in X^\alpha,$
is finite. It follows that any $G_x$-semiinvariant vector in $V$ is
$\g_x$-invariant. Therefore $\dim V^\f\geqslant 1$ for any $\f\in
Y$. By Lemma \ref{Lem:3.2}, $\f$ is spherical for any $\f\in Y$
whence $\dim V^{\f}\leqslant 1$. Thus the map $\psi: Y\rightarrow
\P(V), \f\mapsto V^\f,$ is well-defined. Let us check that this map
is a morphism of varieties. Let $Z$ denote the subvariety of $\g^d,
d=\dim\h_1,$ consisting of all linearly independent $d$-tuples and
$\pi:Z\twoheadrightarrow \Gr_d(\g)$ be the natural projection.
Choose a basis $v^1,\ldots,v^n\in V^*$. 
 Since $\dim
V^\f=1$ for all $\f\in Y$, we have the natural morphism
$\widetilde{\psi}:\pi^{-1}(Y)\rightarrow \P(V)\cong \Gr_{\dim
V-1}(V^*)$ mapping $(\xi_1,\ldots,\xi_d)$ to the linear span of
$\xi_i v^j, i=\overline{1,d},j=\overline{1,n}$. Now recall that
$\pi:\pi^{-1}(Y)\rightarrow Y$ is the quotient morphism for the
natural action $\GL_d:\pi^{-1}(Y)$ and
$\widetilde{\psi}:\pi^{-1}(Y)\rightarrow \P(V)$ is
$\GL_d$-invariant. Therefore $\widetilde{\psi}$ factors through a
unique morphism $\psi:Y\rightarrow \P(V)$. Clearly,
$\psi\circ\delta_X|_{\overline{X}^\alpha}=\varphi|_{\overline{X}^\alpha}$.

Now consider the general case.  Since $X$ is critical for $\alpha$,
we have $\Pi=\Supp(\alpha)\cup \Pi^a$. By Lemma 6.1 from
\cite{Knop_conj}, $\Pi=\Supp(\alpha)$. Then, inspecting Table 1 in
\cite{Wasserman}, we see that $(\g,\h_1,\alpha)$ is one of the
following triples:
\begin{enumerate}
\item $\g=\sl_{n+1},n\geqslant 2, \h_1=\gl_n,
\alpha=\alpha_1+\ldots+\alpha_n$.
\item $\g=\so_{2n+1},n\geqslant 2, \h_1=\gl_n\ltimes \bigwedge^2
\K^2$, $\alpha=\alpha_1+\ldots+\alpha_n$.
\item $\g=\sp_{2n}, n\geqslant 2, \h_1=(\sp_{2n-2}\times\so_2)\ltimes \K, \alpha=\alpha_1+2\alpha_2+\ldots+
2\alpha_{n-1}+\alpha_n$.
\item $\g=G_2, \h_1=(\t_1\times\sl_2)\ltimes (\K^2\oplus\K)$, where
$\t_1$ is a one-dimensional reductive subalgebra of $\g$,
$\alpha=\alpha_1+\alpha_2$.
\end{enumerate}
Note that $\codim_{\h_1}[\h_1,\h_1]=1$ and $[\h_1,\h_1]$ is a
spherical subalgebra of $\g$. In all cases $\h_0$ is the kernel of
$\alpha$ in a certain parabolic subalgebra of $\g$ containing
$\b^-$. Since $\alpha\not\in \Pi$, we get
$\codim_{\h_0}[\h_0,\h_0]=1$. Analogously to the previous paragraph,
 the map
$\widetilde{\psi}:\overline{G[\h_1,\h_1]}\rightarrow \P(V)$ mapping
$\f\in \overline{G[\h_1,\h_1]}$ to $V^\f$ is a morphism. Since
$\codim_{\h_0}[\h_0,\h_0]=\codim_{\h_1}[\h_1,\h_1]$, we see that
there is a (unique) $G$-equivariant map $\iota:Y\rightarrow
\overline{G[\h_1,\h_1]}$ mapping $\h_1$ to $[\h_1,\h_1]$ and $\h_0 $
to $[\h_0,\h_0]$. Using a technique similar to that from the
previous paragraph, we get that $\iota$ is a morphism. It remains to
put $\psi=\widetilde{\psi}\circ\iota$.
\end{proof}

The idea to use results of \cite{Pezzini} in the proof of smoothness
of Demazure's embedding is due to Knop.

So it remains to consider the case when $X$ is critical for
$\alpha\in \Pi\cap\Psi_{G,X}$. Suppose at first that $\rank_G(X)=2$.
Let $D^+,D^-$ denote different elements of $\D_{G,X}(\alpha)$ and
$\beta$ a unique element of $\Psi_{G,X}\setminus\{\alpha\}$. Since
$\alpha$ is not distinguished, we have $\langle
\varphi_{D^+},\beta\rangle\neq \langle\varphi_{D^-},\beta\rangle$.
It was essentially proved by Luna in \cite{Luna_Dem}, Assertion 4,
that in this case the restriction of $d_z\delta_X$ to
$(T_zX)_\alpha$ is injective (note that $c.[X^\gamma]$ in (**) in
the proof of Assertion 4 is equal to
$\langle\varphi_{D^\pm},\gamma\rangle$, thanks to Lemmas 3.2.3,3.3
from \cite{Luna4}).

We are going to reduce the general case to the previous one. Choose
$\beta\in \Psi_{G,X}$ with $\langle\beta,\varphi_{D^+}\rangle\neq
\langle \beta,\varphi_{D^-}\rangle$.  Then
$\langle\varphi_{D_0^\pm},\beta\rangle=\langle\varphi_{D^\pm},\beta\rangle$
for different elements $D_0^\pm\in
\D_{G,\overline{X}^{\{\alpha,\beta\}}}(\alpha)$ (see \cite{Luna5},
the last paragraph of Subsection 3.2). Clearly, $(T_zX)_\alpha=
(T_z\overline{X}^\alpha)_\alpha=
(T_z\overline{X}^{\{\alpha,\beta\}})_\alpha$. Choose $x\in X^\alpha$
and put $\h_0=\delta_X(x),
\widetilde{\h}_0=\delta_{\overline{X}^{\{\alpha,\beta\}}}(x)$. It
follows from the properties of the Demazure morphisms quoted in
Section \ref{SECTION_prelim} that $\h_0$ is an ideal of
$\widetilde{\h}_0$ and $\widetilde{\h}_0/\h_0$ is a commutative
diagonalizable Lie algebra (of dimension $\rank_G(X)-2$).

Let $Q$ denote the parabolic subgroups of $G$ containing
 $B^-$ corresponding to the subset $\Pi^a\sqcup\{\alpha\}\subset
\Pi$. Let $\q_0$ denote a unique ideal in $\q$ complimentary to
$\g_\alpha$ and $Q_0$ be the connected subgroup of $Q$ with Lie
algebra $\q_0$. Then $\overline{X}^\alpha$ is $G$-equivariantly
isomorhic to $G*_QX_\alpha$, where $Q$ acts on $X_\alpha$ via the
projection $Q\twoheadrightarrow Q/Q_0$. As a $G_\alpha$-variety
$X_\alpha$ is isomoprhic to $\P^1\times\P^1$. Note that
$\h_0,\widetilde{\h}_0$ are ideals in $\g_x$ of codimension
$\rank_G(X),2,$ respectively.

Since $\delta_{\overline{X}^{\{\alpha,\beta\}}}$ is an isomorphism,
we have $G_x=N_G(\widetilde{\h}_0)$. Let $\widetilde{H}_0, H_0$
denote the connected subgroups of $G$ with  Lie algebras
$\widetilde{\h}_0,\h_0$. From  Lemma \ref{Lem:3.2} it follows that
$\h_0$ is a spherical subalgebra of $\g$. It follows that
$N_G(\h_0)/H_0$ is a commutative group. In particular,
$\widetilde{H}_0/H_0$ commutes with $N_G(\h_0)/H_0$ whence
$N_G(\h_0)\subset N_G(\widetilde{\h}_0)=G_x$. On the other hand,
$G_x\subset N_G(\h_0)$, for $\delta_X$ is $G$-equivariant. So  the
restriction of $\delta_X$ to $X^\alpha$ is injective.

Now we apply the argument from \cite{Luna_Dem}, proof of Assertion
4. Choose $x\in X_\alpha^\varnothing$. By above, $\g_x=\t+\q_0$. Let
$\m$ denote a unique Levi subalgebra of $\q$ containing $\t$. Set
$\m_0:=\m\cap\q_0$. Since $\g_x\subset \n_\g(\delta_X(x))$ and
$\g_x/\delta_X(x)$ is a diagonalizable Lie algebra, we see that
$\underline{\q}:=\Rad_u(\q)+[\m_0,\m_0]\subset \delta_X(x)$.  So we
may consider $\delta_X|_{X_\alpha}$ as a morphism to
$\q/\underline{\q}$. The Lie algebra $\q/\underline{\q}$ is
identified with $\m_1:=\g_\alpha\oplus \z(\m_0)$. Let $M_1$ denote
the connected subgroup of $G$.  As we have shown above,
$N_{M_1}(\delta_X(x))$ is a maximal torus of $M_1$ for any $x\in
X_\alpha^\varnothing$. Analogously to the proof of Assertion 4 in
\cite{Luna_Dem}, $\overline{M_1\delta_X(x)}$ is smooth. So
$\delta_X$ is an isomorphism of $X_\alpha$ to
$\overline{M_1\delta_X(x)}$ whence its restriction to
$(T_xX)_\alpha$ is injective.

\bigskip

{\Small Chair of Higher Algebra, Department of Mechanics and
Mathematics, Moscow State University.

E-mail address: ivanlosev@yandex.ru}

\begin{thebibliography}{99}
\bibitem[BP]{BP} P. Bravi, G. Pezzini. {\it Wonderful varieties of type
$D$}. Preprint (2004), arXiv:math/RT.0410472, 60 pages.
\bibitem[Br]{Brion} M. Brion. {\it Vers une g\'{e}n\'{e}ralisation des espaces
symm\'{e}triques}. J. Algebra, 134(1990), 115-143.
\bibitem[CP]{CP} C. De Concini, C. Procesi. {\it Complete symmetric varieties,
I}. Invariant theory, Proceedings (F. Gherardelly, ed.) Lect. Notes
in Math., v. 996, 1-44.   Montecatini, 1983, Springer-Verlag.
\bibitem[D]{Dem} M. Demazure. {\it Limites de groupes ortogonaux ou
symplectiques}. Preprint (1980), Paris.
\bibitem[K]{Knop8} F. Knop. {\it Automorphisms, root systems and compactifications}.
J. Amer. Math. Soc. 9(1996), n.1, p. 153-174.
\bibitem[Lo1]{Knop_conj} I.V. Losev. {\it Proof of the Knop
conjecture}. Preprint(2006) arXiv:math.AG/0612561v4, 20 pages.
\bibitem[Lo2]{unique} I.V. Losev. {\it Uniqueness property for spherical
homogeneous spaces}. Preprint (2007).
\bibitem[Lu1]{Luna4} D. Luna. {\it Grosses cellules pour les
vari\'{e}t\'{e}s sph\'{e}riques}. Austr. Math. Soc. Lect. Ser., v.9,
267-280. Cambridge University Press, Cambridge, 1997.
\bibitem[Lu2]{Luna5} D. Luna. {\it Vari\'{e}t\'{e}s sph\'{e}riques
de type A}. IHES Publ. Math., 94(2001), 161-226.
\bibitem[Lu3]{Luna_Dem} D. Luna. {\it Sur le plongements de
Demazure}. J. of Algebra, 258(2002), p. 205-215.
\bibitem[OV]{VO} A.L. Onishchik, E.B. Vinberg.
 {\it Seminar on Lie groups and algebraic groups}. Moscow, Nauka 1988 (in Russian).
 English translation: Berlin, Springer, 1990.
\bibitem[P]{Pezzini} G. Pezzini. {\it Simple immersions of wonderful
varieties}. Preprint (2005), arXiv:math.AG/0506661.
 \bibitem[T]{Timashev_rev} D.A. Timashev, {\it Homogeneous spaces and
equivariant embeddings}. Preprint (2006), arXiv:math.AG/0602228.
 \bibitem[V]{Vust1} Th. Vust. {\it Op\'{e}ration de groupes r\'{e}ductifs dans
un type de c\^{o}nes presque homog\`{e}nes}. Bull. Soc. Math.
France, 102(1974), 317-334.
\bibitem[W]{Wasserman} B. Wasserman. {\it Wonderful varieties of rank two}.
Transform. Groups, v.1(1996), no. 4, p. 375-403.
\end{thebibliography}
\end{document}